\documentclass{amsart}
\usepackage{amssymb,euscript,amsmath, mathrsfs}
\usepackage[dvips]{graphicx}
\usepackage[dvips]{color}

\newcounter{ENUM}

\input xy
\xyoption{all}
\CompileMatrices

\def\<{\langle}
\def\>{\rangle}
\def\0{{{\bf 0}}}

\def\OO{{\mathcal O}}

\def\FA{{\mathfrak A}}

\def\FF{{\mathbb F}}

\def\QQ{{\mathbb Q}}

\def\pp{{\mathfrak p}}

\def\tV{{\tilde V}}

\def\tpi{{\tilde \pi}}
\def\tchi{{\tilde \chi}}
\def\ttheta{{\tilde \theta}}
\def\tpsi{{\tilde \psi}}
\def\tphi{{\tilde \phi}}
\def\tPsi{{\tilde \Psi}}

\def\tLambda{{\tilde \Lambda}}

\def\txi{{\tilde \xi}}
\def\tomega{{\tilde \omega}}
\def\tDelta{{\tilde \Delta}}

\def\trho{{\tilde \rho}}

\newcommand{\Ind}{\operatorname{Ind}}
\newcommand{\cInd}{\operatorname{c-Ind}}
\newcommand{\Res}{\operatorname{Res}}
\newcommand{\Ad}{\operatorname{Ad}}

\def\q{{\mathfrak q}}

\def\Tor{\operatorname{Tor}}

\def\Hom{\operatorname{Hom}}

\def\Fr{\operatorname{Fr}}

\def\End{\operatorname{End}}

\def\Gal{\operatorname{Gal}}
\def\GL{\operatorname{GL}}

\def\tr{\operatorname{tr}}
\def\Spec{\operatorname{Spec}}

\def\Res{\operatorname{Res}}

\def\univ{\operatorname{univ}}

\def\ord{\operatorname{ord}}

\newcommand{\Id}{\operatorname{Id}}

\newcommand{\margh}[1]{}

\newtheorem{thm}{Theorem}[section]
\newtheorem{prop}[thm]{Proposition}
\newtheorem{lemma}[thm]{Lemma}
\newtheorem{cor}[thm]{Corollary}

\theoremstyle{definition}
\newtheorem{defn}[thm]{Definition}

\newtheorem{rem}[thm]{Remark}

\numberwithin{equation}{section}

\begin{document}
\title{On $l$-adic families of cuspidal representations of $\GL_2(\QQ_p)$}
\author{David Helm}

\maketitle

\begin{abstract}
We compute the universal deformations of cuspidal representations $\pi$
of $\GL_2(F)$ over $\overline{\FF}_l$, where $F$ is a local field of
residue characteristic $p$ and $l$ is an odd prime different from $p$.
When $\pi$ is supercuspidal there is an irreducible,
two dimensional representation $\rho$ of $G_F$ that corresponds to
$\pi$ via the mod $l$ local Langlands correspondence of~\cite{viglanglands}; we show that there is
a natural isomorphism between the universal deformation rings of
$\rho$ and $\pi$ that induces the usual (suitably normalized) local Langlands
correspondence on characteristic zero points.  Our work establishes certain
cases of a conjecture of Emerton~\cite{emerton}.
\end{abstract}

\section{Introduction}

A standard technique in the theory of automorphic forms is to study their
variation in algebraic families, and to compare such families of forms to
families of Galois representations.  On the other hand, many questions about
automorphic forms are much more naturally phrased in the language of
automorphic representations.  It is thus natural to try to study the variation
of automorphic representations in families.  Unfortunately, at present there is 
no good notion of automorphic representations over 
rings other than fields of characteristic zero, and hence no good notion of
a family of automorphic representations.

If one instead works in a local setting, say with admissible representations of reductive
groups over a local field $F$, the situation is much better.  Indeed, 
Vign{\'e}ras \cite{vigbook} has extensively studied the representations of such groups
over general coefficient rings.  Moreover, in the case of $\GL_n$ she obtains a
mod $l$ local Langlands correspondence for $GL_n(F)$~\cite{viglanglands}, when the 
characteristic of $F$
is different from $l$, that is compatible in a certain sense with ``reduction mod $l$''
from the classical (characteristic zero) local Langlands correspondence.

Vign{\'e}ras' results provide a framework in which one can study the deformation
theory of mod $l$ admissible representations of $\GL_n(F)$.  It is then natural
to ask if one can extend the mod $l$ local Langlands correspondence to deformations;
that is: if $\pi$ is the admissible representation associated to a representation
$\rho$ of $W_F$ by mod $l$ local Langlands, then is there a natural bijection between
deformations of $\rho$ and those of $\pi$?  Such a bijection would amount to a ``local
Langlands correspondence in families.''

We address this question in the case when
$\pi$ is a supercuspidal representation of $\GL_2(F)$ over $\overline{\FF}_l$
(so $\rho: W_F \rightarrow \GL_2(\overline{\FF}_l)$ is irreducible.)  In this case
$\pi$ admits a universal deformation, which we compute via the theory of types.  Using
the explicit description of local Langlands due to Bushnell-Henniart~\cite{BH}, we
show that there is a natural isomorphism between the universal deformation ring of
$\pi$ and the universal deformation ring of $\rho$ (see Theorem \ref{thm:supercuspidal}
for the precise statement.)  This gives a bijection between deformations of $\rho$
over a complete Noetherean local $W(\overline{\FF}_l)$-algebra $A$ and
deformations of $\pi$ over $A$, extending the ordinary local Langlands correspondence.

The question of constructing a local Langlands correspondence of $\GL_2$ in algebraic
families has also been considered in~\cite{emerton}.  Emerton considers families of
two-dimensional representations $\rho$ of $G_F$ over reduced complete $l$-torsion free
Noetherean local $W(k)$-algebras $A$, and shows that for such a representation
there is at most one admissible $A[\GL_2(F)]$-module $\pi$ such that:
\begin{itemize}
\item $\pi$ is $A$-torsion free,
\item at generic points $\eta$ of $\Spec A$, the smooth dual $\pi_{\eta}^{\vee}$ of
$\pi_{\eta}$ corresponds to $\rho_{\eta}$ under a certain ``modified local Langlands 
correspondence'', and
\item at the closed point $x$ of $\Spec A$, $\pi_x$ is a quotient of the 
$\pi(\rho_x)^{\vee}$, where $\pi(\rho_x)$ is the representation attached to $\rho_x$
by this modified local Langlands correspondence.
\end{itemize}
(We refer the reader to Theorems 1.2 and 1.3 of \cite{emerton} for details.)
Emerton conjectures further that such a $\pi$ always exists; our result establishes
this in the case when $\rho_x$ is absolutely irreducible.  (More precisely,
in this setting $\rho$ is a deformation of $\rho_x$; Theorem~\ref{thm:supercuspidal}
gives a corresponding deformation of $\pi(\rho_x)$, and the representation $\pi$ 
conjectured by Emerton is the smooth dual of this deformation.)

The structure of the paper is as follows: we begin by studying the deformation
theory of a representation of a finite group over a finite field $k$ of 
characteristic $l$, and show that a very naive approach
to constructing the universal deformation (that proceeds essentially by concatenating
the characteristic zero lifts of this representation) actually produces a deformation
that ``agrees with the universal deformation up to $l$-torsion.''  We give a
criterion in terms of character theory for this naive deformation to be universal.
The next section applies this theory for mod $l$ representations of $\GL_2(\FF_q)$,
where $l$ is odd and prime to $q$.

In section~\ref{sec:cuspidal} we turn to the deformation theory of cuspidal
representations $\pi$ of $\GL_2(F)$ over $\overline{\FF}_l$, where $F$ is a local field 
with residue field $\FF_q$.  This proceeds by the theory of types; we show that
deforming such a representation is equivalent to deforming its type.  We then use this
to attach to a supercuspidal $\pi$ a character $\chi$ of either $F^{\times}$ or $E^{\times}$, where $E$ 
is a quadratic extension of $F$, in such a way that deformations of $\pi$ are naturally
in bijection with deformations of $\chi$.  (See Theorem~\ref{thm:admissible} for details.)  
We also compute the universal deformations of representations of $\pi$ that are cuspidal
but not supercuspidal; although this is not necessary for our main results it is of
a piece conceptually with the ideas described here, and has applications to Emerton's
conjecture that we will address in a future paper.

The deformation theory of two-dimensional irreducible representations of $G_F$ is
well-understood, and we summarize the necessary facts in section~\ref{sec:galois}.
Combining this with the explicit local Langlands for $\GL_2$ of \cite{BH} establishes
Theorem \ref{thm:supercuspidal}.

We would like to thank Richard Taylor and Matthew Emerton for many valuable discussions,
and Jared Weinstein for his comments on an earlier version of this paper.

\section{Deformation Theory} \label{sec:def}
We begin with some general preliminaries on the deformation theory of representations
of a finite group $G$.  Let $k$ be a finite field of characteristic $l$ and let $\pi$ be
an absolutely irreducible $k$-representation of $G$.  Attached to $\pi$ we have
the universal deformation ring $R_{\pi}^{\univ}$ of $\pi$; it parameterizes lifts
of $\pi$ to representations over finite length local $W(k)$-algebras with residue field $k$.
The ring $R_{\pi}^{\univ}$ is a complete Noetherian local $W(k)$-algebra, with residue field $k$,
and comes equipped with is a universal representation $\pi^{\univ}$ of $G$ over 
$R_{\pi}^{\univ}$, and an isomorphism:
$$\iota: \pi^{\univ} \otimes_{R_{\pi}^{\univ}} k \cong \pi.$$
We refer the reader to~\cite{mazur} for more details.

Under certain hypotheses that we explain further below, it is possible to compute the ring 
$R_{\pi}^{\univ}$ directly from the character table of $G$. We begin by constructing
a particular deformation of $\pi$:

Let $K$ be an algebraic closure of the field of fractions of $W(k)$, and let $\OO$ be
the ring of integers in $K$.  We say a $K$-reprsentation $\tpi$ of $G$ lifts $\pi$ if
there is a $G$-stable $\OO$-lattice $\tpi_{\OO}$ in $\tpi$ such that
$\tpi_{\OO} \otimes_{\OO} \overline{k}$ is isomorphic to $\pi \otimes_k \overline{k}$.
One sees easily that the set $S$ of isomorphism classes of such $\tpi$ is in bijection 
with the set of $K$-points of $\Spec R_{\pi}^{\univ}$.  

For $\tpi$ in $S$, let $K_{\tpi}$ be the minimal field of definition of $\tpi$.  Define:
$$R_K := \prod_{\tpi \in S} K_{\tpi}$$
$$\pi_K := \prod_{\tpi \in S} \tpi.$$
In the definition of $\pi_K$ we consider each $\tpi$ to be a representation over $K_{\tpi}$,
so that $\pi_K$ is an $R_K$-representation of $G$.

Each $\tpi$ arises by base change from $\pi^{\univ}$ via a canonical map
$R_{\pi}^{\univ} \rightarrow K_{\tpi}$; together these give a canonical map
$R_{\pi}^{\univ} \rightarrow R_K$.  Let $R$ be its image, and let
$\pi_R$ be the representation $\pi^{\univ} \otimes_{R_{\pi}^{\univ}} R$.
We have $\pi_K \cong \pi_R \otimes_R R_K$.

On the other hand, let $R_0$ be the $W(k)$-subalgebra of $R$ generated by
the traces of elements $\sigma$ of $G$ on $\pi_R$.  By~\cite{Ca}, Theorem 2,
$\pi_R$ is defined over $R_0$.  Thus the surjection
$$R_{\pi}^{\univ} \rightarrow R$$
factors through $R_0$.  We must therefore have $R_0 = R$.  

One thus has an explicit description of $R$ in terms of
the character table of $G$: for each $\sigma$ in $G$, let $x_{\sigma}$
be the element of $R_K$ defined by:
$$(x_{\sigma})_{\tpi} = \tr \sigma|_{\tpi}.$$
Then $x_{\sigma}$ is the trace of $\sigma$ on $\pi_R$, and hence $R$ is
simply the $W(k)$-algebra generated by the $x_{\sigma}$.  

\begin{rem} \rm
The closed immersion $\Spec R \rightarrow \Spec R_{\pi}^{\univ}$
induced by the above surjection is a bijection on $K$-points.  One can
show that the nilpotents of $R_{\pi}^{\univ}$ are supported on
the special fiber of $\Spec R_{\pi}^{\univ}$; it follows that
after inverting $l$ the map $R_{\pi}^{\univ} \rightarrow R$
is an isomorphism.  In other words $R$ is simply the quotient of
$R_{\pi}^{\univ}$ by its $l$-power torsion.  We will not make use of
this, however.
\end{rem}

We will give criteria for the natural map $R_{\pi}^{\univ} \rightarrow R$
to be an isomorphism.  For each $\tpi$, let $e_{\tpi}$ be the idempotent
in $K[G]$ corresponding to $\tpi$.  Also note that for any finite length
local $W(k)$-algebra $A$ with residue field $k$, and any deformation $\pi_A$ of
$\pi$, the center $Z(W(k)[G])$ of $W(k)[G]$ acts on $\pi_A$ by scalars because
of Schur's lemma.  We thus obtain a natural map $Z(W(k)[G]) \rightarrow A$.

\begin{thm} \label{thm:defthy}
Suppose that the following conditions both hold:
\begin{enumerate}
\item The sum $e_{\pi} := \sum_{\tpi \in S} e_{\tpi}$ lies in $\OO[G]$, and
\item For any finite length $W(k)$-algebra $A$ with residue field $k$, and any
deformation $\pi_A$ of $\pi$, the $W(k)$-subalgebra $A_0$ of $A$ generated by the
traces $\tr \sigma|_{\pi_A}$ for $\sigma \in G$ is contained in the image of
the natural map $Z(W(k)[G]) \rightarrow A$.
\end{enumerate}
Then the natural map $R_{\pi}^{\univ} \rightarrow R$ is an isomorphism.
\end{thm}
\begin{proof}
The second condition implies that the natural map $Z(W(k)[G]) \rightarrow R$
is surjective; i.e. $R$ is generated by the endomorphisms of $\pi_R$ arising
from class functions in $W(k)[G]$.  Let $a_1, \dots, a_r$ be such class functions,
and let $P \in W(k)[x_1, \dots, x_r]$ be a polynomial such that $P(a_1, \dots, a_r)$
annihilates $\pi_R$.  We will show that $P(a_1, \dots, a_r)$ annihilates any $A$-deformation
$\pi_A$ of $\pi$, over any $A$.  

It suffices to show this after making a base change from $W(k)$ to $\OO$. 
Note that $e_{\pi}$ is the identity on $\tpi$ for any $\tpi$, and hence $e_{\pi}$
is the identity on $\pi$.  By Nakayama's lemma it follows that $e_{\pi}\pi_A = \pi_A$.
It thus suffices to show that $e_{\pi}P(a_1, \dots, a_r)$ is zero in $\OO[G]$,
or even in $K[G]$.  As a $K[G]$ module, $e_{\pi}K[G]$ is a direct sum of copies
of various $\tpi \in S$.  Each such $\tpi$ arises by base change from $\pi_R$
and hence is annihilated by $P(a_1, \dots, a_r)$, as required.

It follows that any such $P(a_1, \dots, a_r)$ annihilates $\pi^{\univ}$.
We thus obtain a well defined map
$$R \rightarrow R_{\pi}^{\univ}$$
that, given an element of $R$, lifts it to an element of $Z(W(k)[G])$
and sends it to the corresponding element of $R_{\pi}^{\univ}$.  It suffices
to show this map is surjective.  But by~\cite{Ca}, Theorem 2, $R_{\pi}^{\univ}$ is
generated by the traces of elements of $G$ on $\pi^{\univ}$, and hence
(by condition 2) by the image of $Z(W(k)[G])$ in $R_{\pi}^{\univ}$.  Since
this is in the image of $R$ the result follows.
\end{proof}

\section{Representations of $\GL_2(\FF_q)$}
We now apply the results of the previous section to the group $G = \GL_2(\FF_q)$, for $q = p^r$.
We first recall some basic facts about the representation theory of $G$ (see for instance~\cite{BH}, chapter 6).  
Let $B$ be the 
standard Borel subgroup of $G$, $N$ its unipotent radical, and $T$ the standard torus in $B$.
Finally, let $Z$ denote the center of $G$.

Let $K$ be an algebraic closure of $\QQ_l$, and consider
two characters $\tphi_1,\tphi_2: \FF_q^{\times} \rightarrow K^{\times}$.  We can view
the pair $\tphi_1,\tphi_2$ as a character of $T$ in the obvious way, and extend it to
a character $\tphi$ of $B$ trivial on $N$.  Then the representation $\Ind_B^G \tphi$ is irreducible
of dimension $q + 1$ if $\tphi_1 \neq \tphi_2$.  If $\tphi_1 = \tphi_2$, then
$\Ind_B^G \tphi$ decomposes as a direct sum of the character $\tphi_1 \circ \det$ 
with an irreducible ``Steinberg representation'' of dimension $q$.

Those irreducible $K$-representations of $G$ that do not occur as subquotients of some
$\Ind_B^G \tphi$ are called
supercuspidal, and can be constructed as follows: Let $\tpsi: N \rightarrow K^{\times}$
be a nontrivial character.  Fix a subgroup $E$ of $G$ isomorphic to $\FF_{q^2}^{\times}$,
and a character $\ttheta: E \rightarrow K^{\times}$, such that $\ttheta^q \neq \ttheta$.
Consider the character $\ttheta_{\tpsi}$ of $ZN$ defined by $\ttheta_{\tpsi}(zu) = \ttheta(z)\tpsi(u)$
for $z \in Z$, $u \in N$.  Then one can compute the character of the virtual representation 
$$\pi_{\ttheta} = \Ind_{ZN}^G \ttheta_{\tpsi} - \Ind_E^G \ttheta.$$
One has:
$$
\begin{array}{lcll}
\tr \pi_{\ttheta}(z) &=& (q - 1)\ttheta(z) & \mbox{for $z \in Z$}\\
\tr \pi_{\ttheta}(zu) &=& -\ttheta(z) & \mbox{for $z \in Z$, $u \in N \setminus \{\Id\}$}\\
\tr \pi_{\ttheta}(t) &=& 0 & \mbox{for $t \in T \setminus Z$}\\
\tr \pi_{\ttheta}(y) &=& -\ttheta(y) - \ttheta^q(y) & \mbox{for $y \in E \setminus Z$.}
\end{array}
$$
This exhausts the conjugacy classes of $G$, and one easily verifies that $\pi_{\ttheta}$
is an irreducible representation of $G$ of dimension $q - 1$.  Moreover, $\pi_{\ttheta}$
is independent of $\tpsi$.  We have $\pi_{\ttheta} = \pi_{\ttheta^q}$, and 
$\pi_{\ttheta} \neq \pi_{\ttheta^{\prime}}$ unless $\ttheta^{\prime} \in \{\ttheta,\ttheta^q\}$.
A simple count then shows that every supercuspidal $K$-representation of $G$ is of this form.

We also observe that the representation $\Ind_N^G \tpsi$ contains every irreducible representation
of $G$ except for those of the form $(\tphi_1 \circ \det)$, with multiplicity one (c.f.~\cite{BH}, 
p. 48.)

From this information it is straightforward to compute the irreducible representations of
$G$ over $\overline{\FF}_l$, for $l$ odd and prime to $p$.  (All that is necessary is modular character
theory at the level of \cite{serre}, chapter 18.)  This is well-known, but we summarize the
results below for ease of reference:

Since the order of $G$ is $(q^2 - 1)(q^2 - q)$,
if $q$ is not congruent to $\pm 1$ mod $l$ the representation theory over $\overline{\FF}_l$
is ``the same'' as in characteristic zero; reduction mod $l$ takes the character of
a representation in characteristic zero to the character of the corresponding representation
mod $l$, and this correspondence is a bijection.  There are thus two cases of interest to us:
$l$ an odd divisor of either $q - 1$ or $q + 1$.

First assume $q \equiv 1$ mod $l$.  Given 
$\phi_1,\phi_2: \FF_q^{\times} \rightarrow \overline{\FF_l}^{\times}$
we obtain a character $\phi$ of $B$ as above; then $\Ind_B^G \phi$ is irreducible
if $\phi_1 \neq \phi_2$, and splits as a direct sum of a character and a Steinberg representation
otherwise.  

The irreducible representations that do not arise in the above way can be described as follows:
for each character $\theta: \FF_{q^2}^{\times} \rightarrow \overline{\FF_l}^{\times},$
satisfying $\theta^q \neq \theta$, we can lift $\theta$ to a character $\ttheta$ with
values in $K^{\times}$.  It is easy to see by computing the modular character that the mod $l$ 
reduction of $\pi_{\ttheta}$ is an irreducible $\overline{\FF}_l$-representation $\pi_{\theta}$ of 
$G$, that depends only on $\theta$ and not on $\ttheta$.  As in characteristic zero, we have
$\pi_{\theta} = \pi_{\theta^{\prime}}$ if, and only if, $\theta^{\prime} \in \{\theta,\theta^q\}$.
A counting argument shows that every irreducible 
$\overline{\FF_l}^{\times}$-representation of $G$ that does not arise from parabolic induction is
of the form $\pi_{\theta}$.

If $q \equiv -1$ mod $l$, the situation is slightly different.  As in the previous case,
given $\phi_1,\phi_2$ we have $\Ind_B^G \phi$ irreducible if $\phi_1 \neq \phi_2$.  On the other
hand, if $\phi_1 = \phi_2$, the Jordan-Holder constituents of $\Ind_B^G \phi$ consist
of two copies of the character $(\phi_1 \circ \det)$ and an irreducible representation
$\pi_{\phi_1}$ of $G$ of dimension $q - 1$.  The representation $\pi_{\phi_1}$ can
be described as follows: let $\theta: \FF_{q^2}^{\times} \rightarrow \overline{\FF}_l^{\times}$
be defined by $\theta(x) = \phi_1(x^{q+1})$.  Then $\theta^q = \theta$, but we can lift
$\theta$ to a character $\ttheta$ with values in a field of characteristic zero.  Moreover,
since $q \equiv -1$ mod $l$, we can choose $\ttheta$ such that $\ttheta^q \neq \ttheta$.
(This is easily seen to be possible only when $q$ is congruent to $-1$ mod $l$.)  Then
the characteristic zero representation $\pi_{\ttheta}$ reduces mod $l$ to $\pi_{\phi_1}$.  (One
verifies that the mod $l$ reduction of $\pi_{\ttheta}$ is irreducible and that its modular
character agrees with that of $\pi_{\phi_1}$.)
Conversely, if $\theta: \FF_{q^2}^{\times} \rightarrow \overline{\FF}_l^{\times}$ is a character
satisfying $\theta^q = \theta$, then we have $\theta(x) = \phi_1(x^{q+1})$ for some
character $\phi_1$ of $\FF_q^{\times}$; we define $\pi_{\theta}$ to be the representation
$\pi_{\phi_1}$ constructed above.  The representations $\pi_{\theta}$ are not supercuspidal,
as they arise as Jordan-Holder constituents of a representation of the form $\Ind_B^G \phi$,
but they are {\em cuspidal}; i.e., they do not arise as subrepresentations of a representation
of the form $\Ind_B^G \phi$.

As in the $q \equiv 1$ mod $l$ case, the remaining representations of $G$ over $\overline{\FF}_l$
are parametrized by the characters $\theta: \FF_{q^2}^{\times} \rightarrow \overline{\FF}_l^{\times}$
such that $\theta^q \neq \theta$.  Given such a character we define $\pi_{\theta}$ to be
the mod $l$ reduction of $\pi_{\ttheta}$, where $\ttheta$ is any lift of $\theta$ to characteristic
zero.  As before, $\pi_{\theta}$ does not depend on the particular $\ttheta$ chosen, and
is an irreducible supercuspidal representation of $G$ over $\overline{\FF}_l^{\times}$.

Now fix a particular $q$, $l$, and $\theta: \FF_{q^2}^{\times} \rightarrow \overline{\FF}_l^{\times}$.
We require that $\theta$ admits a lift $\ttheta$ to characteristic zero such that
$\ttheta^q \neq \ttheta$; in this case the representation $\pi_{\theta}$ is well-defined.
Let $k$ be a finite field of characteristic $l$ over whice $\pi_{\theta}$ is defined.
Our main goal is to apply the theory of the previous section to the deformation
theory of $\pi_{\theta}$.  

Let $S_{\theta}$ denote the set of equivalence classes of characters 
$\ttheta: \FF_{q^2}^{\times} \rightarrow K^{\times}$ such that $\ttheta^q \neq \ttheta$
and the mod $l$ reduction of $\ttheta$ is $\theta$.
We consider two such characters $\ttheta,\ttheta^{\prime}$ equivalent if $\ttheta^q = \ttheta^{\prime}$,
and use the notation $[\ttheta]$ for the equivalence class of $\ttheta$.
(Note that this equivalence relation is trivial unless $\theta^q = \theta$ and $l$ divides $q+1$.)  
Then the correspondence $\ttheta \mapsto \pi_{\ttheta}$
defines a bijection between the equivalence classes in $S_{\theta}$ and the lifts of $\pi_{\theta}$
to representations over $K$.  

\begin{prop} The conditions of Theorem~\ref{thm:defthy} apply to $\pi_{\theta}$.
\end{prop}
\begin{proof}
For each $\ttheta \in S_{\theta}$, let $e_{\ttheta}$
be the idempotent in $K[G]$ corresponding to $\pi_{\ttheta}$.
The idempotent $e_{\pi_{\theta}}$ appearing in condition 1) of Theorem~\ref{thm:defthy} is given
by:
$$e_{\pi_{\theta}} = \sum_{\ttheta \in S_{\theta}} e_{\ttheta}.$$
It follows directly from our calculation of the characters of the representations $\pi_{\ttheta}$
that $e_{\pi_{\theta}}$ lies in $\OO[G]$, so condition 1) is satisfied.

We now turn to condition 2) of Theorem~\ref{thm:defthy}.  Note that it suffices to verify this condition
after base change from $W(k)$ to $W(k^{\prime})$, for $k^{\prime}$ a finite extension of $k$.  We may
therefore assume we have a nontrivial character $\Psi$ of $N$ with values in $k^{\times}$;
$\Psi$ lifts uniquely to a character $\tPsi$ with values in $W(k)^{\times}$.  We consider $\Psi$
and $\tPsi$ as characters of $ZN$ that are trivial on $Z$.  We may also assume that
all $n$th roots of unity are in $k$ for $n | q^2 - 1$, and $n$ prime to $l$.  (In other words,
that all characters $E \rightarrow \overline{\FF_l}^{\times}$ take values in $k^{\times}$.)
Let $\chi$ be the restriction
of $\theta$ to the center $Z$ of $G$, and extend it to a character of $ZN$ trivial on $N$.
Note that $\Res^G_B \pi_{\theta} = \Ind_{ZN}^B \chi\Psi$.

Let $A$ be a finite length $W(k)$-algebra with residue field $k$,
and let $\pi_A$ be an $A$-deformation of $\pi_{\theta}$.  Let $\chi_A$ be the
central character of $\pi_A$.  
By definition, $Z$ acts on $\pi_A$ by $\chi_A$; since $N$ has order prime to $l$
and $\pi_{\theta}$ contains a subspace on which $N$ acts by $\chi$, $\pi_A$ contains
a $N$-stable direct summand on which $N$ acts by $\tPsi$.  Thus $\pi_A$
contains a $ZN$-subrepresentation isomorphic to $\chi_A\tPsi$, so we obtain
a nonzero map $\Ind_{ZN}^B \chi_A\tPsi \rightarrow \Res^G_B \pi_A$.  This map
reduces modulo the maximal ideal $m_A$ of $A$ to the corresponding map 
$\Ind_{ZN}^B \chi\Psi \rightarrow \pi_{\theta}$, which is an isomorphism of
$B$-representations.  Thus $\Res^G_B \pi_A$ is isomorphic to $\Ind_{ZN}^B \chi_A\tPsi$.

One sees easily from this that for any $\sigma \in B$, the trace of $\sigma$ on $\pi_A$
is in the subalgebra of $A$ generated by $\chi_A(Z)$, and hence in the image of the
map $Z(W(k)[G]) \rightarrow A$.  Since any $\sigma$ in $G$ is conjugate either to
an element of $B$ or an element of $E \setminus Z$, it remains to verify
condition 2) for elements of $E \setminus Z$.

Suppose first that $l$ does not divide $q-1$.  Let $\sigma$ be an element of
$E \setminus Z$, and note that there are $q(q-1)$ elements conjugate to $\sigma$.
The element:
$\sum_{\sigma^{\prime} \sim \sigma} \sigma^{\prime}$
of $Z(W(k)[G])$ acts on $\pi_A$ with trace $q(q-1)\tr(\sigma)$ on $\pi_A$.
On the other hand, it acts on $\pi_A$ by a scalar $c$ so its trace is $(q-1)c$.
Thus $\tr(\sigma)$ is $\frac{c}{q}$ which indeed lies in the image of $Z(W(k)[G])$ in $A$.

Now suppose that $l$ does not divide $q+1$.  In this case $\Res^G_E \pi_{\theta}$
is a sum of characters $\theta^{\prime}: E \rightarrow k^{\times}$, where
$\theta^{\prime}$ agrees with $\chi$ on $Z$ and $\theta^{\prime} \notin \{\theta,\theta^q\}.$
There are exactly $q-1$ such $\theta^{\prime}$ and each occurs with multiplicity one.  Moreover,
each such $\theta^{\prime}$ admits a unique lift to an $A^{\times}$-valued
character $\theta^{\prime}_A$ such that $\theta^{\prime}_A(\sigma) = \chi_A(\sigma)$
for $\sigma \in Z$.  Then $\Res^G_E \pi_A$ is the direct sum of these $\theta^{\prime}_A$.

It follows that for $\sigma \in E$, the trace of $\sigma$ on $\chi_A$ is
given by $-\theta_A(\sigma) - \theta_A^q(\sigma)$, where $\theta_A$ is the unique lift
of $\theta$ that agrees with $\chi_A$ on $Z$.  Since $l$ does not divide $q+1$, we
can write $\sigma = \tau\tau^{\prime}$ where $\tau$ is in $Z$ and $\tau^{\prime}$
has order prime to $l$.  Then $\theta_A(\tau^{\prime})$ is just the Teichmuller lift
of $\theta(\tau^{\prime})$, and therefore lies in the image of $W(k)$ in $A$.
On the other hand, $\theta_A(\tau)$ is equal to $\chi_A(\tau)$, and hence lies in
the image of $Z(W(k)[G])$ in $A$.
\end{proof}

This gives us an explicit description of the universal deformation ring
$R_{\pi_{\theta}}^{\univ}$ of $\pi_{\theta}$.  In particular, the maps
$$f_{\ttheta}: R_{\pi_{\theta}}^{\univ} \rightarrow K$$
arising from the representations $\pi_{\ttheta}$ for $\ttheta \in S_{\theta}$
identify $R_{\pi_{\theta}}^{\univ}$ with the W(k)-subalgebra of $\prod_{[\ttheta] \in S_{\theta}} K$
generated by the elements $x_{\sigma}$ for $\sigma \in G$, where
$$(x_{\sigma})_{[\ttheta]} = \tr \sigma|_{\pi_{\ttheta}}.$$

Explicitly, we have:
$$
\begin{array}{lcll}
(x_{\sigma})_{[\ttheta]} &=& (q-1)\ttheta(\sigma) & \mbox{for $\sigma \in Z$}\\
(x_{\sigma\tau})_{[\ttheta]} &=& -\ttheta(\sigma) & \mbox{for $\sigma \in Z$, $\tau \in N \setminus \{\Id\}$}\\
(x_{\sigma})_{[\ttheta]} &=& 0 & \mbox{for $\sigma \in T \setminus Z$}\\
(x_{\sigma})_{[\ttheta]} &=& -\ttheta(\sigma) - \ttheta^q(\sigma) & \mbox{for $\sigma \in E \setminus Z$.}
\end{array}
$$

Suppose $\theta \neq \theta^q$.  Then the equivalence relation on $S_{\theta}$ is trivial, so we can
simply write $\ttheta$ in place of $[\ttheta]$ when discussing elements of $S_{\theta}$.
Define elements $y_{\sigma}$ of $\prod_{\ttheta \in S_{\theta}} K$
by $(y_{\sigma})_{\ttheta} = \ttheta(\sigma)$ for $\sigma \in E$.  It is clear that $R_{\pi_{\theta}}^{\univ}$
is contained in the W(k)-subalgebra of $\prod_{\ttheta \in S_{\theta}} K$ generated by the $y_{\sigma}$.  On
the other hand, this subalgebra is simply $R_{\theta}^{\univ}$, the universal deformation ring
of $\theta$.  We thus obtain a map 
$$R_{\pi_{\theta}}^{\univ} \rightarrow R_{\theta}^{\univ}.$$

\begin{thm} \label{thm:lvl0} The map
$$R_{\pi_{\theta}}^{\univ} \rightarrow R_{\theta}^{\univ}$$
is an isomorphism.  In particular there is a natural bijection between
$A$-deformations of $\pi_{\theta}$ and $A$-deformations of $\theta$
that induces the correspondence $\ttheta \mapsto \pi_{\ttheta}$
on lifts of $\theta$ to $K^{\times}$.
\end{thm}
\begin{proof}
It suffices to show that each $y_{\sigma}$ is in the $W(k)$-algebra generated by the
$x_{\sigma}$.  This is clear for $\sigma \in Z$; each such $y_{\sigma}$ is $-x_{\sigma\tau}$
for a nontrivial $\tau$ in $N$.  

As $\theta$ takes values in a field of characteristic $l$, it has order prime to $l$.  Thus $\theta$
has a unique lift $\ttheta_l$ to with order prime to $l$; $\ttheta_l$ takes values in $W(k)^{\times}$. 
If we let $E^l$ be the subgroup of $E$ of order prime to $l$, then for any
$\ttheta$ in $S$, $\ttheta_l^{-1}\ttheta$ is trivial on $E^l$.  Note that
since $\theta^{q-1}$ is a nontrivial character, so is $\ttheta_l^{\q-1}$.

Then for $\sigma \in E \setminus Z$ we have:
$$\sum_{\tau \in E^l} \ttheta_l^{-1}(\tau)[\ttheta(\sigma\tau) + \ttheta^q(\sigma\tau)]
= \sum_{\tau \in E^l} \ttheta(\sigma) + \ttheta_l^{q-1}(\tau)\ttheta^q(\sigma)
= (\# E_l^{\times}) \ttheta(\sigma).$$
In terms of the $y_{\sigma}$ and $x_{\sigma}$ this becomes:
$$y_{\sigma} = \frac {1}{\# E^l} \sum_{\tau \in E_l^{\times}} \ttheta_l^{-1}(\tau) x_{\sigma}$$
for all $\sigma \in E \setminus Z$.  As $\ttheta_l$ takes values in $W(k)^{\times}$ we are
done.
\end{proof}

We now turn to the case where $\theta = \theta^q$.  In this case $q$ is congruent to $-1$ mod $l$,
and the equivalence relation on the $\ttheta$ in $S$ is nontrivial.  For $\sigma \in E$,
define an element $y_{\sigma}$ of $\prod_{[\ttheta] \in S_{\theta}} K$ by
$$(y_{\sigma})_{[\ttheta]} = \ttheta(\sigma) + \ttheta(\sigma)^q.$$
(This is clearly independent of the choice of $\ttheta$ representing $[\ttheta]$.)
For $\sigma \in \FF_q^{\times}$, $y_{\sigma}$ is simply an element of $W(k)^{\times}$.
In particular the elements $y_{\sigma}$ lie in the $W(k)$-subalgebra of $\prod_{[\ttheta] \in S_{\theta}} K$
generated by the $x_{\sigma}$; that is, we can consider the $y_{\sigma}$ as elements
of $R_{\pi_{\theta}}^{\univ}$.  Moreover, each $x_{\sigma}$ has a simple expression in terms of
the $y_{\sigma}$.  Thus $R_{\pi_{\theta}}^{\univ}$ is the subalgebra of $\prod_{[\ttheta] \in S_{\theta}} K$
generated by the $y_{\sigma}$.

As in the $\theta \neq \theta^q$ case, let $\ttheta_l$ be the unique lift of $\theta$ of order prime to $l$.
(Note that since $\theta = \theta^q$, $\ttheta_l = \ttheta_l^q$ and so is {\em not} in $S_{\theta}$.)
Then $\ttheta_l$ takes values in $W(k)^{\times}$, and any lift $\ttheta$ of $\theta$ agrees with $\ttheta_l$
on $E_l^{\times}$.  Thus for any $\tau \in E_l^{\times}$, $y_{\sigma\tau}$ differs from $y_{\sigma}$ by
a scalar factor in $W(k)^{\times}$.  In particular $R_{\pi_{\theta}}^{\univ}$ is generated by
$y_{\sigma}$ for $\sigma$ of $l$-power order.  In fact, one verifies easily that if $\sigma$ generates
the pro-$l$ part of $E$, then the single element $y_{\sigma}$ generates all of $R_{\pi_{\theta}}^{\univ}$.
We will find it slightly more convenient to use the generator $y_{\sigma} - 2$ instead;
the minimal polynomial of $y_{\sigma} - 2$ is the polynomial $Q$ defined by
$$Q(t)^2 = \prod_{\zeta^{l^n} = 1; \zeta \neq 1} (t - \zeta - \zeta^{-1} + 2),$$
where $n = \ord_l(q^2 - 1).$ We thus have:

\begin{thm} \label{thm:cusp1}
Suppose $q$ is congruent to $-1$ mod $l$, and $\theta^q = \theta$.  Then for any choice
of generator $\sigma$ of the pro-$l$ part of $E$ there is an isomorphism
$$W(k)[[t]]/Q(t) \rightarrow R_{\pi_{\theta}}^{\univ}$$
sending $t$ to $y_{\sigma} - 2$.  Under this isomorphism, the trace of $\sigma$
on $\pi_{\theta}^{\univ}$ is $-t - 2$.
\end{thm}

Thus, for any complete Noetherean local $W(k)$-algebra $A$, and any $\alpha$ in $A$
with $Q(\alpha) = 0$, there is a unique $A$-deformation $\pi_{\theta,\alpha}$ of $\pi_{\theta}$
for which $\tr \pi_{\theta,\alpha}(\sigma) = -\alpha - 2$.
 
\section{Cuspidal representations of $\GL_2(F)$} \label{sec:cuspidal}

We now turn to the deformation theory of cuspidal representations.  Fix a $p$-adic field
$F$, with residue field $\FF_q$, and a prime $l$ different from $p$.  Our main goal
will be to compute universal deformation rings of irreducible cuspidal representations $\pi$
of $\GL_2(F)$ over $\overline{\FF}_l$.  

\begin{defn} Let $\pi$ be an irreducible admissible representation of $\GL_2(F)$
over $\overline{\FF}_l$, and let $A$ be a finite length local $W(\FF_l)$-algebra,
with maximal ideal $m_A$.  An $A$-deformation of $\pi$ is
a free $A$-module $\pi_A$ with an action of $\GL_2(F)$, together with an isomorphism 
$\pi_A/m_A \pi_A \cong \pi,$ such that $\pi_A$ is smooth; i.e every element of
$\pi_A$ is fixed by a compact open subgroup of $\GL_2(F)$. 
\end{defn}
Note that if $U$ is a compact open subgroup of $\GL_2(F)$, of order prime to $l$,
then taking $U$-invarants is exact; it follows that $\pi_A^U$ will be finitely
generated for all $U$ of this sort.  Since such $U$ are cofinal in the set of
all compact open subgroups of $\GL_2(F)$,
$\pi_A$ will be admissible as a $\GL_2(F)$-module. 

The key tool we will use to compute the deformation theory of representations $\pi$ is
the theory of types.  In characteristic $l$ this theory is due to Vign{\'e}ras \cite{vigbook}.
What we need will follow most easily from Bushnell-Henniart's very explicit description
of the types of cuspidal $\overline{\QQ}_l$-representations of $\GL_2(F)$.  

Fix a character $\tpsi$ of $F$ of level one, with values in $\overline{QQ}_l^{\times}$.  Then for any simple 
stratum $(\FA,n,\alpha)$ in $M_2(F)$, we have a character $\tpsi_{\alpha}$ of $U_{\FA}^{[\frac{n}{2}]+1}$
defined by
$$\tpsi_{\alpha}(x) = \tpsi(\tr(\alpha (x-1)).$$
(c.f.~\cite{BH}, 12.5).  Let $\psi_{\alpha}$ be the mod $l$ reduction of $\tpsi_{\alpha}$.  We then have:

\begin{thm}[\cite{BH}, 15.5]
Let $\tpi$ be an irreducible, cuspidal representation of $\GL_2(F)$ over $\overline{\QQ}_l$.  
Then $\tpi = (\cInd_J^{\GL_2(F)} \tLambda^{\prime}) \otimes (\chi \circ \det)$, where $\chi$ is a character
of $F^{\times}$ and either:
\begin{enumerate}
\item $J = F^{\times}\GL_2(\OO_F)$, and
the restriction of $\tLambda^{\prime}$ to $\GL_2(\OO_F)$ is the inflation of an irreducible, cuspidal
representation of $\GL_2(\FF_q)$, or
\item there exists a stratum $(\FA, n, \alpha)$ such that $n$ is odd, 
$J = E^{\times}U_{\FA}^{\frac{n+1}{2}}$, and $\tLambda^{\prime}$ is an irreducible representation
of $J$ whose restriction of to $U_{\FA}^{\frac{n+1}{2}}$ is the character $\tpsi_{\alpha}$, or
\item there exists a stratum $(\FA, n, \alpha)$ such that $n$ is even,
$J = E^{\times}U_{\FA}^{\frac{n}{2}}$, and $\tLambda^{\prime}$ is an irreducible representation of $J$
whose restriction to $U_{\FA}^{\frac{n}{2}+1}$ is a multiple of $\tpsi_{\alpha}$.
\end{enumerate}
Here $E$ is the field $F[\alpha]$; it is a quadratic extension of $F$.
\end{thm}

Results of Vign{\'e}ras (\cite{viglanglands}, Theorems 1.1 and 1.2) allow us to turn the above description
into a characterization of supercuspidal representations of $\GL_2(F)$ over $\overline{\FF}_l$.
In particular, every such representation $\pi$ arises by reduction mod $l$ from a supercuspidal 
representation $\tpi$ over $\overline{\QQ}_l$ whose central character is integral.  More precisely there 
is a unique homothety class of lattices in such a $\tpi$, and the mod $l$ reduction of any such lattice is 
isomorphic to $\pi$. 
Thus if we set $\tpi = (\cInd_J^{\GL_2(F)} \tLambda^{\prime}) \otimes (\chi \circ \det),$ we can find
a $\GL_2(F)$-stable lattice in $\tpi$ by choosing a $J$-stable lattice in $\tLambda^{\prime} \otimes (\chi \circ \det)$.
Reducing this lattice mod $l$ we find:

\begin{cor} \label{cor:modl}
Let $\pi$ be an irreducible, supercuspidal representation of $\GL_2(F)$ over $\overline{\FF}_l$.  
Then $\pi = (\cInd_J^{\GL_2(F)} \Lambda^{\prime}) \otimes (\phi \circ \det)$, where $\phi$ is a character
of $F^{\times}$ and either:
\begin{enumerate}
\item $J = F^{\times}\GL_2(\OO_F)$, and
the restriction of $\Lambda^{\prime}$ to $\GL_2(\OO_F)$ is the inflation of an irreducible, supercuspidal
representation of $\GL_2(\FF_q)$, or
\item there exists a stratum $(\FA, n, \alpha)$ such that $n$ is odd, 
$J = E^{\times}U_{\FA}^{\frac{n+1}{2}}$, and $\Lambda^{\prime}$ is a character
of $J$ whose restriction to $U_{\FA}^{\frac{n+1}{2}}$ is $\psi_{\alpha}$, or
\item there exists a stratum $(\FA, n, \alpha)$ such that $n$ is even,
$J = E^{\times}U_{\FA}^{\frac{n}{2}}$, and $\Lambda^{\prime}$ is an irreducible representation of $J$
whose restriction to $U_{\FA}^{\frac{n}{2}+1}$ is a multiple of $\psi_{\alpha}$.
\end{enumerate}
\end{cor}

Fix a representation $\pi$ as in the corollary.  It will be convenient to work with
a certain subgroup $U$ of $J$, defined as follows: in case $1$, $U$ is the kernel of the
map: $\GL_2(\OO_F) \rightarrow \GL_2(\FF_q)$.  In cases $2$ and $3$, $U$ is
$U_{\FA}^{[\frac{n+1}{2}]}$.  Let $\Lambda = \Lambda^{\prime} \otimes (\phi \circ \det)$,
and let $V$ be the subspace of $\pi$ on which $U$ acts via $\Lambda.$

As $J$ normalizes $\Res^J_U \Lambda$, $V$ is a $J$-stable subspace of $\pi$.  
In fact, we have:

\begin{lemma} As representations of $J$, we have $V \cong \Lambda$.
\end{lemma}
\begin{proof}
Clearly $\Lambda$ is a $J$-subrepresentation of $\pi$, and hence also of $V$.  On
the other hand, as a representation of $U$, $V$ is a direct sum of copies of
$\Res^J_U \Lambda.$  By Mackey's induction-restriction formula (\cite{vigbook}, I.5.5), we have:
$$\Res^{\GL_2(F)}_U \pi = \oplus_{JgU} 
\cInd_{U \cap gJg^{-1}}^U (\Res^J_U \Lambda)^g.$$
For each $g$, let $U_g = U \cap gJg^{-1}$.  Then we have
$$\Hom_U(\cInd_{U_g}^U (\Res^J_U \Lambda)^g, \Res^J_U \Lambda) =
\Hom_{U_g}(\Res^J_{U_g} (\Lambda)^g, \Res^J_{U_g} \Lambda),$$
and the latter is nonzero if and only if $g$ intertwines $\Lambda$.  By~\cite{BH}, 15.1, $g$ is 
then an element of $J$.  Thus $\pi$ contains exactly one copy of $\Res^J_U \Lambda$, so
$\Lambda$ must be all of $V$ as required.
\end{proof}

\begin{prop}
Let $A$ be a finite length $W(\overline{\FF}_l)$-algebra $A$, and let $\pi_A$ be a lift of $\pi$
to a representation over $A$.  The map
$$\Lambda_A \mapsto \cInd_J^{\GL_2(F)} \Lambda_A$$
gives a bijection between
$A$-deformations of $\Lambda$ and $A$-deformations of $\pi$.
\end{prop}
\begin{proof}
Let $\pi_A$ be an $A$-deformation of $\pi$.
Since $U$ is a $p$-group, the representation
$\Res^J_U \Lambda$ lifts {\em uniquely} to $A$; let $\tV$ be the $A$-submodule
of $\pi_A$ on which $U$ acts via this lift.  As an $A$-module $\tV$ is a direct summand of 
$\pi_A$; its reduction mod $l$ is $V$.  Moreover, $\tV$ is stable under $J$, and
therefore defines an irreducible representation $\Lambda_A$ of $J$ lifting
$\Lambda$.  We have a natural map 
$$\cInd_J^{\GL_2(F)} \Lambda_A \rightarrow \pi_A$$
that reduces modulo $m_A$ to the isomorphism
$$\cInd_J^{\GL_2(F)} \Lambda \rightarrow \pi.$$
As $\pi_A$ is a free $A$-module it follows by the long exact sequence for $\Tor$
that this map is injective; on the other hand, Nakayama's lemma for admissible 
representations (see for instance~\cite{emerton}, Lemma 2.7) shows that
the image of this map is all of $\pi_A$.  
\end{proof}

The deformations of $\Lambda$ are not difficult to understand, as $\Lambda$
is ``not far from'' a pro-$p$ group.  We will be most interested in the case in which
$\Lambda$ arises from an {\em admissible pair} (mod $l$) (c.f.~\cite{BH}, 5.19).

\begin{defn} (c.f.~\cite{BH}, 18.2).  An admissible pair (over $\overline{\FF}_l$)
is a pair $(E,\chi)$, where $E$ is a tamely ramified quadratic extension of $F$,
and $\chi: E^{\times} \rightarrow \overline{\FF}_l^{\times}$, such that:
\begin{itemize}
\item $\chi$ does not factor through the norm $N_{E/F}: E^{\times} \rightarrow F^{\times}$.
\item Let $U^1_E$ be the group of units in $\OO_E$ congruent
to $1$ mod $\pp$, where $\pp$ is a uniformizer of $\OO_E$.  If the restriction of $\chi$
to $U^1_E$ factors through $N_{E/F}$, then $E$ is unramified over $F$.
\end{itemize}

Two pairs $(E,\chi)$ and $(E^{\prime},\chi^{\prime})$ are isomorphic if there
is an isomorphism of $j: E \rightarrow E^{\prime}$ such that $\chi^{\prime} \circ j = \chi$.
\end{defn}

We now describe a ``mod $l$'' version of the parametrization of tame cuspidal representations,
which associates a supercuspidal representation to an admissible pair.  We follow~\cite{BH}, 5.19;
little needs to be changed, but we will need the explicit description of the parameterization
in order to properly understand the deformation theory.

Given $(E,\chi)$, we first choose a character $\phi$ of $F^{\times}$ such that
$\chi = (\phi \circ N_{E/F}) \chi^{\prime}$, with $\chi^{\prime}$ a {\em minimal}
character of $E^{\times}$ in the sense of~\cite{BH}, 5.18.  There are then three cases:

\begin{enumerate}
\item $\chi^{\prime}$ has level 0.  In this case $E/F$ is unramified, and the
restriction of $\chi^{\prime}$ to $\OO_E^{\times}$ is inflated from a character
$\theta$ of $\FF_{q^2}^{\times}$, that satisfies $\theta^q \neq \theta$.  The
representation $\pi_{\theta}$ of $\GL_2(\FF_q)$ then inflates to a representation
of $\GL_2(\OO_F)$; we extend this to a representation $\Lambda^{\prime}$ of
$J = F^{\times} \GL_2(\OO_F)$ be letting $F^{\times}$ act via $\chi^{\prime}$.
We take $\Lambda_{\chi} = \Lambda^{\prime} \otimes (\phi \circ \det)$,
and $\pi_{\chi} = \cInd_J^{\GL_2(F)} \Lambda_{\chi}$.

\item $\chi^{\prime}$ has odd level $n$.  Fix a stratum
$(\FA,n,\alpha)$, and let $J = E^{\times}U_{\FA}^{\frac{n+1}{2}}$.  We take
$\Lambda^{\prime}$ to be the character of $J$ whose restriction to $E^{\times}$ is
$\chi$ and whose restriction to $U_{\FA}^{\frac{n+1}{2}}$ is $\psi_{\alpha}$.  As
above, we set $\Lambda_{\chi} = \Lambda^{\prime} \otimes (\phi \circ \det)$, and
$\pi_{\chi} = \cInd_J^{\GL_2(F)} \Lambda_{\chi}$.

\item $\chi^{\prime}$ has even positive level $n = 2m$.  In this case $E/F$ is unramified.
Fix a stratum $(\FA, n, \alpha)$, and let $J = E^{\times}U_{\FA}^m$.
Let $\OO_E^{\times,1}$ be the subgroup of $\OO_E^{\times}$ of units that map to
$1$ in $\FF_{q^2}$, and set $J^1 = \OO_E^{\times,1}U_{\FA}^m$.  In~\cite{BH}, 19.5.4,
Bushnell-Henniart associate to this data an irreducible representation $\eta$ of $J^1$
depending on $\chi$, and show that there is a unique irreducible representation
$\Lambda^{\prime}$ of $J$ such that:
\begin{itemize}
\item $\Lambda^{\prime} | J^1 = \eta$,
\item $\Lambda^{\prime} | F^{\times}$ is a multiple of $\chi | F^{\times}$, and
\item for every root of unity $\zeta$ in $E^{\times} \setminus F^{\times}$, of
order prime to $p$, the trace of $\Lambda^{\prime}(\zeta)$ is $-\chi(\zeta)$.
\end{itemize}
We take $\Lambda_{\chi} = \Lambda^{\prime} \otimes (\phi \circ \det)$,
and $\phi_{\chi} = \cInd_J^{\GL_2(F)} \Lambda_{\chi}$.
\end{enumerate}

\begin{rem} \rm The construction described above is precisely the mod $l$ reduction
of the construction of Bushnell-Henniart; that is, if $(E,\tchi)$ is a characteristic
zero admissible pair whose mod $l$ reduction is $(E,\chi)$, then the characteristic
zero representation $\pi_{\tchi}$ associated to $(E,\tchi)$ by Bushnell-Henniart
reduces mod $l$ to the representation $\pi_{\chi}$ described above.
\end{rem}

It is fairly straightforward to extend this construction to deformations.
We will need the following lemma:

\begin{lemma}
Let $J$ be a locally profinite group and let $J^1$ be a compact normal subgroup.
Let $A$ be a finite length local $W(\overline{\FF}_l)$-algebra, and
let $\eta_A$ be a finite-dimensional representation of $J^1$
over $A$, such that $\eta_A \otimes_A \overline{\FF}_l$ is irreducible.
Suppose $\eta_A$ extends to an $A$-representation $\Lambda_A$ of $J$.
Then any other extension $\Lambda^{\prime}_A$ of $\eta_A$ to
a representation of $J$ differs from $\Lambda_A$ by twisting
by a character of $J/J^1$.
\end{lemma}
\begin{proof}
Let $V$ and $V^{\prime}$ be the representation spaces of $\Lambda$ and
$\Lambda^{\prime}$ respectively; fix a $J^1$-equivariant isomorphism
between them, so that we can consider $\Lambda_A(\tau)$ and $\Lambda^{\prime}_A(\tau)$
as elements of $\End(V)$ for any $\tau \in J$.  Consider the element
$\Lambda_A(\tau)^{-1}\Lambda^{\prime}_A(\tau)$ of $\End(V)$.  This element
commutes with $\Lambda_A(\sigma)$ for any $\sigma \in J^1$, and hence
is an $\eta_A$-equivariant endomorphism of $V$.  By Schur's lemma
$\Lambda_A(\tau)^{-1}\Lambda^{\prime}_A(\tau)$ is equal to
$c_{\tau}$ times the identity for a scalar $c_{\tau} \in A^{\times}$.
One verifies easily that $c_{\tau} = 1$ if $\tau \in J^1$ and
that $c_{\tau\tau^{\prime}} = c_{\tau} c_{\tau^{\prime}}$ for all
$\tau$, $\tau^{\prime}$.  In particular $\tau \mapsto c_{\tau}$ is
a character of $J/J^1$, and twisting $\Lambda_A$ by this character yields
$\Lambda^{\prime}_A$.
\end{proof}

\begin{thm} \label{thm:admissible}
Let $(E,\chi)$ be a (mod $l$) admissible pair.  Then for any
finite length $W(\overline{\FF}_l)$-algebra $A$, there are natural
bijections between $A$-deformations $\chi_A$ of $\chi$, deformations
$\Lambda_{\chi,A}$ of $\Lambda_{\chi}$, and deformations $\pi_{\chi,A}$
of $\pi_{\chi}$.
\end{thm}
\begin{proof}
We have already constructed the bijection between deformations of $\Lambda_{\chi}$
and $\pi_{\chi}$, so it remains to relate these to deformations of $\chi$.  We proceed case-by-case.
Fix a $\chi^{\prime}$ and $\phi$ such that $\chi = \chi^{\prime} (\phi \circ N_{E/F})$ with
$\chi^{\prime}$ minimal.
\begin{enumerate}
\item Suppose that $\chi^{\prime}$ has level zero.  Given a lift $\chi^{\prime}_A$
of $\chi^{\prime}$, the restriction of $\chi^{\prime}_A$ to $\OO_E^{\times}$ is 
inflated from a character $\theta_A$
of $\FF_{q^2}^{\times}$ lifting $\theta$.  Then $\pi_{\theta_A}$ is an $A$-deformation
of $\pi_{\theta}$; inflating it to $\GL_2(\OO_F)$ and extending it to
$F^{\times} \GL_2(\OO_F)$ by letting $F^{\times}$ act by $\chi^{\prime}_A$ gives
a lift $\Lambda^{\prime}_A$ of $\Lambda^{\prime}$.  Conversely, given a lift
$\Lambda^{\prime}_A$ of $\Lambda^{\prime}$, the restriction of $\Lambda^{\prime}_A$
to $\GL_2(\OO_F)$ is inflated from an $A$-deformation $\pi_{\theta,A}$ of $\pi_{\theta}$;
by Theorem~\ref{thm:lvl0} $\pi_{\theta,A}$ arises from a (uniquely determined)
character $\theta_A$ of $\FF_{q^2}^{\times}$ lifting $\theta$.  Let $\chi^{\prime}_A$
be the unique $A$-deformation of $\chi^{\prime}$ that is inflated from $\theta_A$
on $\OO_E^{\times}$ and agrees with $\Lambda^{\prime}_A$ on $F^{\times}$.  This
recovers $\chi^{\prime}_A$ from $\Lambda^{\prime}_A$ and gives a bijection
between deformations of $\chi^{\prime}$ and $\Lambda^{\prime}$; twisting
we obtain the desired bijection between deformations of $\chi$ and deformations
of $\Lambda_{\chi}$.

\item If $\chi^{\prime}$ has odd level $n$, then $\Lambda^{\prime}$ is a
character of $E^{\times} U_{\FA}^{\frac{n+1}{2}}$ whose restriction to $E^{\times}$ 
is $\chi^{\prime}$, and whose restriction to the pro-$p$ group $U_{\FA}^{\frac{n+1}{2}}$
is a fixed character $\psi_{\alpha}$.  Given a deformation $\chi^{\prime}_A$
of $\chi^{\prime}$ we let $\Lambda^{\prime}_A$ be the character
whose restriction to $E^{\times}$ is $\chi^{\prime}_A$ and whose restriction
to $U_{\FA}^{\frac{n+1}{2}}$ is the unique lift of $\psi_{\alpha}$ to $A$.
This manifestly yields a bijection between deformations of $\chi^{\prime}$ and
deformations of $\Lambda^{\prime}$; twisting we obtain a bijection between
deformations of $\chi$ and deformations of $\Lambda_{\chi}$.

\item Suppose $\chi^{\prime}$ has even level $2m$.  Then $\Lambda^{\prime}$
is a representation of $J = E^{\times} U_{\FA}^m$; its restriction
to the normal subgroup $J^1$ of $J$ given by $J^1 = \OO_E^{\times,1} U_{\FA}^m$
is an irreducible representation $\eta$.  If we lift $\chi^{\prime}$ to a
character $\tchi^{\prime}$ with values in $W(\overline{\FF}_l)^{\times}$ (which
we can always do), then $(E,\tchi^{\prime})$ is an admissible pair, to which
Bushnell-Henniart associate a representation $\tLambda^{\prime}$ of $J$
over $W(\overline{\FF}_l)$ lifting $\Lambda^{\prime}$.

Suppose we have an $A$-deformation
$\Lambda^{\prime}_A$ of $\Lambda^{\prime}$.  The restriction of $\Lambda^{\prime}_A$
to $J^1$ is the unique $A$-deformation $\eta_A$ of $\eta$, since $J^1$ is a
pro-$p$ group.  By the lemma above, $\Lambda^{\prime}_A$ is a twist of
$\tLambda^{\prime} \otimes_{W(\overline{\FF}_l)} A$ by a uniquely determined character
of $J/J^1$.

Now let $\chi^{\prime}_A$ be an $A$-deformation of $\chi^{\prime}$.  Define
$\Lambda^{\prime}_A$ to be the twist of $\tLambda^{\prime}_A \otimes_{W(\overline{\FF}_l)} A$
by the character $\chi^{\prime}_A (\tchi^{\prime})^{-1}$.  This gives a bijection
between deformations of $\chi^{\prime}$ and deformations of $\Lambda^{\prime}$, and
is independent of the choice of $\tchi^{\prime}$.  As usual, twisting yields the
desired bijection between deformations of $\chi$ and deformations of $\Lambda$.
\end{enumerate}
\end{proof}

If $E/F$ is ramified, there is an even simpler classification of deformations of $\pi_{\chi}$:
\begin{prop} Let $(E,\chi)$ is an admissible pair with $E/F$ ramified, and let $\phi$ be
its restriction to $F^{\times}$.  Then restriction to $F^{\times}$ is a bijection
between $A$-deformations of $\chi$ and $A$-deformations of $\phi$.  In 
particular giving an $A$-deformation of
$\pi_{\chi}$ is equivalent to giving an $A$-deformation of its central character.
\end{prop}
\begin{proof}
Let $\phi_A$ be an $A$-deformation of $\phi$.  The group $\OO_E^{\times}$ is generated
by the roots of unity in $E^{\times}$ (all of which lie in $F^{\times}$), and the pro-$p$
group $\OO_E^{\times,1}$.  Thus $\phi_A$ extends uniquely to a character of
$F^{\times} \OO_E^{\times}$ that lifts $\chi$.  If $\pp$ is a uniformizer of $\OO_E$, then $\pp$ and
$F^{\times} \OO_E^{\times}$ generated $E^{\times}$, and $\pp^2$
lies in $F^{\times} \OO_E^{\times}$, so an extension of $\phi_A$ to a character $\chi_A$ 
of $E^{\times}$ lifting $\chi$ is determined uniquely by a choice of $\phi_A(\pp)$
compatible with the (already determined) value of $\phi_A(\pp^2)$.  In particular
$\phi_A(\pp)$ must be a square root of $\phi_A(\pp)^2$ that reduces to $\chi_A(\pp)$
modulo $l$; as $l$ is odd there is a unique choice of $\phi_A(\pp)$.  Thus
$\phi_A$ extends uniquely to a character $\chi_A$ of $E^{\times}$ lifting $\chi$, as required.
\end{proof}

\begin{rem} \rm We will also be concerned with the deformations of supercuspidal
representations $\pi$ over $\overline{\FF}_l$ that correspond via mod $l$ local
Langlands to primitive representations- that is, to representations of $G_F$ that
are not induced from characters.  These do not come from admissible pairs
via the above construction.  On the other hand,
they only occur when $q$ is a power of $2$, and have a very specific form:
by the construction in the proof of~\cite{BH}, Theorem 50.3, they arise (up to a twist) 
from ramified simple 
strata of odd level; that is, up to twisting every such $\pi$ is of the 
form $\cInd_J^{\GL_2(F)} \Lambda$,
where $\Lambda$ is a character of $J$, $J = E^{\times} U_{\FA}^{\frac{n+1}{2}}$
for some stratum $(U,n,\alpha)$, and $E/F$ is {\em ramified}.  
Deformations of such $\pi$ are in bijection
with deformations of $\Lambda$ and (by the same argument as in case 2 of 
Theorem~\ref{thm:admissible}) therefore in bijection with deformations of
the restriction of $\Lambda$ to $E^{\times}$.  The above proposition then shows that
such deformations are determined by their restriction to $F^{\times}$.
It follows that in this case deforming $\pi$ is equivalent to deforming the central 
character of $\pi$.
\end{rem}

Finally, we will need to understand the deformation theory of a particular
admissible representation that is cuspidal but not supercuspidal.  Such
representations only occur when $q$ is congruent to $-1$ mod $l$, and
are equal to a twist by $(\phi \circ \det)$ of the ``Weil representation''
$\pi(1)$ described in~\cite{vigbook}, II.2.5.  From our perspective,
$\pi(1)$ is most conveniently characterised as follows: let $\pi_1$
be the cuspidal representation of $\GL_2(\FF_q)$ associated to the
trivial character of $\FF_{q^2}^{\times}$, and let $\Lambda$ be
the representation of $J = F^{\times} \GL_2(\OO_F)$ such that
$F^{\times}$ acts trivially and $\GL_2(\OO_F)$ acts via $\pi_1$.  
Then $\pi(1)$ is given by $\Ind_J^{\GL_2(F)} \Lambda$.

By Theorem~\ref{thm:cusp1}, the universal deformation
of $\pi_1$ is defined over the ring $R = W(\overline{\FF}_l)[[t]]/Q(t)$.
Define a representation $\Lambda^{\univ}$ of $J$ over
$R[[x]]$ for which $F^{\times}$
acts via the unramified character that takes a uniformizer
of $\OO_F$ to $1 + x$, and $\GL_2(\OO_F)$ acts
via the universal deformation of $\pi_1$.

\begin{prop} \label{prop:cusp2}
The ring $R_{\pi(1)}^{\univ}$ is isomorphic 
to $W(\FF_l)[[x,t]]/Q(t)$.  The univeral representation
$\pi^{\univ}$ is given by $\Ind_J^{\GL_2(F)} \Lambda^{\univ}$.
Under these identifications, the central character of $\rho_{\pi(1)}^{\univ}$
is the unramified character that takes a uniformizer of $\OO_F$ to $1 + x$.
Moreover, if $U$ is the kernel of the map $\GL_2(\OO_F) \rightarrow \GL_2(\FF_q)$,
then $\GL_2(\FF_q)$ acts on the $U$-invariants of $\rho_{\pi(1)}^{\univ}$ by the
deformation $\pi_{1,t}$ of $\pi_1$ over $R_{\pi(1)}^{\univ}$.
\end{prop}
\begin{proof}
The same argument as in the supercuspidal case shows that giving an
$A$-deformation of $\pi(1)$ is the same as giving an $A$-deformation
of $\Lambda$.  It thus suffices to show that $\Lambda^{\univ}$ is
the universal deformation of $\Lambda$.  By the same argument as
in the level 0 case of Theorem~\ref{thm:admissible}, deforming
$\Lambda$ is the same as giving a deformation of its (trivial) central 
character and a deformation of $\pi_1$.  The representation $\Lambda^{\univ}$
is clearly universal for such a pair of deformations.  Moreover, the
$U$-invariants of $\Lambda^{\univ}$ are by construction $\pi_{1,t}$.
\end{proof}

\section{Representations of $G_F$} \label{sec:galois}
We now turn to the Galois side of the local Langlands correspondence.  We begin 
by studying the first-order deformations of two-dimensional representations 
of $G_F$ over $\overline{\FF}_l$.  Let $\tomega$ be the cyclotomic character
of $G_F$ with values in $W(\FF_l)$, and let $\omega$ be its reduction mod $l$. 

\begin{lemma} \label{lemma:h1}
Let $\rho$ be an irreducible representation of $G_F$ over
$\overline{\FF}_l$.  Suppose $q$ is not congruent to $1$ mod $l$.  Then:
\begin{itemize}
\item $H^1(G_F, \rho)$ is one-dimensional if $\rho$ is trivial
or $\rho = \omega$, and is zero otherwise.
\item $H^2(G_F, \rho)$ is one-dimensional if $\rho = \omega$
and is zero otherwise.
\end{itemize}
If $q$ is congruent to $1$ mod $l$, then:
\begin{itemize}
\item $H^1(G_F, \rho)$ is two-dimensional if $\rho$ is trivial
and is zero otherwise.
\item $H^2(G_F, \rho)$ is one-dimensional if $\rho$ is trivial
and is zero otherwise.
\end{itemize}
\end{lemma}
\begin{proof} This is an easy application of inflation-restriction.
\end{proof}

\begin{prop} \label{prop:galois}
Let $\rho$ be an irreducible two-dimensional representation of $G_F$
over $\overline{\FF}_l$.  Then either:
\begin{itemize}
\item there exists a character $\xi$
of $G_E$, where $E$ is the unique unramified quadratic extension of $F$,
such that $\rho = \Ind_{G_E}^{G_F} \xi$, or
\item the map $\rho_A \mapsto \det \rho_A$ gives a bijection between $A$-deformations $\trho$
of $\rho$ and deformations of $\det \rho$.
\end{itemize}
\end{prop}
\begin{proof}
It suffices to show that if $H^1(G_F, \Ad^0 \rho)$ is nonzero, then
there exists an $\xi$ as above.  By the lemma, this
can only happen if $\Ad^0 \rho$ contains a character as a Jordan-Holder
constituent.  Since $l$ is odd, $\Ad^0 \rho$ is naturally a direct 
summand of $\Ad \rho$ and is therefore a 3-dimensional self-dual representation
of $G_F$.  In particular, if $\Ad^0 \rho$ contains a character as a Jordan-Holder 
consituent then it contains a character $\nu$ as a subrepresentation.  Then
$\nu$ is a subrepresentation of $\Ad \rho$; by Schur's lemma $\nu$
is a nontrivial character of $G_F$.  We therefore have an isomorphism
$\rho \cong \rho \otimes \nu$.  Considering determinants we find that
$\nu^2$ is trivial.  Let $E^{\prime}$ be the quadratic extension of $F$ corresponding to
the kernel of $\nu$.  Then the restriction of $\Ad \rho$ to $G_{E^{\prime}}$
contains two copies of the trivial character, so $\rho$ becomes
reducible when restricted to $G_{E^{\prime}}$.  If we let $\xi^{\prime}$ be a character 
of $G_{E^{\prime}}$ contained in the restriction of $\rho$, then
$\rho \cong \Ind_{G_{E^{\prime}}}^{G_F} \xi^{\prime}$.

We then have a direct sum decomposition: 
$$\Ad^0 \rho = \nu \oplus \Ind_{G_{E^{\prime}}}^{G_F} 
\frac{(\xi^{\prime})^{\sigma}}{\xi^{\prime}}$$
where $\sigma$ generates $\Gal(E^{\prime}/F)$.  If the second direct
summand is irreducible, then $\nu$ must be equal to $\omega$
in order for $\Ad^0 \rho$ to have nontrivial cohomology.  In particular
$E^{\prime}$ is unramified over $F$, so $E^{\prime} = E$.  (Note that then
$q$ must be congruent to $-1$ mod $l$.)

On the other hand, if the second direct summand is reducible, then
$\Ad^0 \rho$ is the direct sum of three characters, which must be
nontrivial by Schur's lemma.  Again, one of these characters must be
$\omega$ in order for $\Ad^0 \rho$ to have nontrivial
cohomology.  It follows that $\rho \cong \rho \otimes \omega$,
and thus (as with $\nu$), $\omega$ has exact order $2$.  The kernel of
$\omega$ is thus $G_E$; it follows that $\rho$ is induced
from a character of $G_E$ as required.
\end{proof}

\begin{prop} \label{prop:induced}
Let $\xi$ be a character of $G_E$, and let 
$\rho = \Ind_{G_E}^{G_F} \xi$.  Suppose $\rho$ is absolutely irreducible.  Then
$$\xi_A \mapsto \Ind_{G_E}^{G_F} \xi_A$$
induces a bijection between $A$-deformations of $\xi$
and $A$-deformations of $\rho$.  (This gives a
natural isomorphism $R_{\xi}^{\univ} \cong R_{\rho}^{\univ}$.)
\end{prop}
\begin{proof}
By induction on the length of $A$, it suffices to prove this for first-order
deformations.  The first order deformations of $\xi$ are a torsor
for $H^1(G_E,1)$, where $1$ denotes the trivial character of $G_E$.  The map
$$H^1(G_E, 1) \rightarrow H^1(G_F,\Ad \rho)$$
that induces a first order deformation of $\xi$ from $G_E$
to $G_F$ is clearly injective.  The dimension of $H^1(G_F,\Ad \rho)$
is two if $\Ad \rho$ contains $\omega$ and one otherwise.
As in the proof of the previous proposition, if $\Ad \rho$ contains
$\omega$ then $\omega$ is trivial on $G_E$.

On the other hand, by the lemma $H^1(G_E,1)$ is two-dimensional if $\omega$
is trivial on $G_E$, and one-dimensional otherwise, so the result follows.
\end{proof}

The deformations of a character of the Galois group of a local field
are easy to describe:
let $\xi$ be a character of $G_F$ with values in $\overline{\FF}_l^{\times}$,
and let $\txi: G_F \rightarrow W(\overline{\FF}_l)^{\times}$ be its Teichmuller lift.
Then $R_{\xi}^{\univ}$ is isomorphic to $W(\overline{\FF}_l)[[t]][\zeta]/\<\zeta^{l^n} - 1\>,$
where $n$ is equal to $\ord_l(q-1)$.  Define a deformation $\xi^{\univ}$ of $\xi$
by $\xi^{univ} = \txi\xi^{\prime}$, where $\xi^{\prime}$ is the character
that takes a Frobenius element $\Fr$ to $1 + t$ and a generator $\sigma$ of the $l$-part
of $I_F$ to $\zeta$.  Local class field theory easily shows that $\xi^{\univ}$ is 
the universal deformation of $\xi$.

\section{The deformation-theoretic correspondence} \label{sec:correspondence}
We are now in a position to relate the deformation theory of an admissible
representation $\pi$ over a finite field of characteristic $l$ to the
representation $\rho$ attached to $\pi$ by the local Langlands correspondence.
Throughout we use the Tate normalization for the local Langlands correspondence;
this has the advantage that $\pi$ and $\rho$ have the same field of definition.
This normalization differs from the more usual normalization by a twist by
an unramified character; it has the property that the central character of $\pi$
corresponds to $\tomega \det \rho$ via local class field theory.

Since we are primarily interested in deformation theory, we will work
with continuous representations of $G_F$, rather than the Weil-Deligne
representations usually considered in the Langlands correspondence, as the latter
do not behave well from a deformation-theoretic perspective.  Over
$\overline{\FF}_l$ this makes no difference, as any Weil-Deligne representation
arises from a unique representation of $G_F$ in this setting.  In
characteristic zero this is not true, but we will (somewhat abusively)
say that an admissible representation $\tpi$ and a continuous Galois
representation $\trho$ ``correspond under local Langlands'' if the Weil-Deligne
representation arising from $\trho$ corresponds to $\tpi$ under the more usual notion of the
local Langlands correspondence.  (This, of course, does {\em not} give
a bijection between irreducible Galois representations and cuspidal admissible
representations, as not every Weil-Deligne representation comes from a representation
of $G_F$.  On the other hand, {\em integral} irreducible representations of $W_F$ over
a local field of residue characteristic $l$- those that contain a lattice-
do arise from representations of $G_F$.)

Bushnell-Henniart give the following description of the characteristic
zero local Langlands correspondence for representations arising from admissible
pairs: (for details, see~\cite{BH}, section 34.)
If an admissible representation $\tpi$ over $\overline{\QQ_l}$ arises from an admissible 
pair $(E,\tchi)$ then the corresponding Galois representation $\trho$ is equal to 
$\Ind_{G_E}^{G_F} \txi$ for some character $\txi$ of $G_E$.  Moreover every such $\trho$ 
arises in this way.  If $E/F$ is unramified then there is a character $\tDelta$ of 
$E^{\times}$, independent of $\tchi$, such that $\txi$ corresponds to $\tchi\tDelta$ 
via class field theory.  (If $E/F$ is ramified, there is still an explicit
description of $\txi$ in terms of $\tchi$, but it is more complicated and depends on 
$\tchi$; we will not make use of this.)

The mod $l$ local Langlands correspondence is compatible with reduction mod $l$ for cuspidal
representations, and so is the process that associates a representation of $\GL_2(F)$ to an
admissible pair.  Thus if $\pi$ is a mod $l$ admissible representation attached to an
admissible pair $(E,\chi)$, we can choose a lift $(E,\tchi)$ to characteristic zero; this
gives a characteristic zero representation $\tpi$ lifting $\pi$.  The Galois representation
$\trho$ corresponding to $\tpi$ is induced from a character $\txi$ of $G_E$ with values in 
$\overline{\QQ}_l^{\times}$.  Then the representation $\rho$ attached to $\pi$ is just the
reduction mod $l$ of $\trho$, so $\rho$ is induced from a character $\xi$ of $G_E$ with
values in $\overline{\FF}_l^{\times}$.  In this manner we see that the statements of
the previous paragraph hold mutatis mutandis for the mod $l$ local Langlands correspondence.

This allows us to prove:
\begin{thm} \label{thm:supercuspidal}
Let $\pi$ be a supercuspidal representation of $\GL_2(F)$ over a finite field of
characteristic $l$ and let $k$ be its field of definition.  Let $\rho$ be the
Galois representation corresponding to $\pi$ under mod $l$ local Langlands.
Then there is a unique isomorphism of $W(k)$-algebras:
$$R_{\pi}^{\univ} \rightarrow R_{\rho}^{\univ}$$
with the property that the induced map
$R_{\pi}^{\univ}(\overline{\QQ}_l) \rightarrow R_{\rho}^{\univ}(\overline{\QQ}_l)$
is the characteristic zero local Langlands correspondence.
\end{thm}
\begin{proof}
Such an isomorphism is clearly unique if it exists.
We begin by working with deformations over $W(\overline{\FF}_l)$ rather than
over $W(k)$ so we can apply the results of section~\ref{sec:cuspidal}.

First assume $\pi$ comes from an admissible pair $(E,\chi)$ with $E/F$ unramified.
Then by Theorem~\ref{thm:admissible} we have an isomorphism of $W(\overline{\FF}_l)$-algebras
$R_{\pi}^{\univ} \cong R_{\chi}^{\univ}$.  On the other hand, $\rho$ is the representation
$\Ind_{G_E}^{G_F} \xi$, where $\xi$ corresponds to $\chi\Delta$ via local class field
theory, and $\Delta$ is the mod $l$ reduction of $\tDelta$.  By Proposition~\ref{prop:induced},
$R_{\rho}^{\univ}$ is isomorphic over $W(\overline{\FF}_l)$ to $R_{\xi}^{\univ}$.
Twisting by $\tDelta$ and applying the local class field theory isomorphism gives
an isomorphism of $R_{\chi}^{\univ}$ with $R_{\xi}^{\univ}$; composing these isomorphisms
gives the desired isomorphism of $R_{\pi}^{\univ}$ with $R_{\rho}^{\univ}$.  That
the induced map on $\overline{\QQ}_l$-points is the characteristic zero local Langlands
correspondence is clear from the construction.

Otherwise, $\pi$ either comes from an admissible pair $(E,\chi)$ with $E/F$ ramified
or does not come from an admissible pair at all.  In either case deformations of
$\pi$ correspond to deformations of the central character of $\pi$.  By 
Proposition~\ref{prop:galois} deformations of the representation $\rho$ correspond
to deformations of its determinant.  In particular for any $A$-deformation $\pi_A$
of $\pi$, there is a unique $A$-deformation $\rho_A$ of $\rho$ such that
the central character of $\pi_A$ corresponds to $\tomega \det \rho_A$ under
local class field theory.  This gives the desired isomorphism of $R_{\pi}^{\univ}$
with $R_{\rho}^{\univ}$.  Let $\tpi$ be a $\overline{\QQ}_l$-point of
$R_{\pi}^{univ}$; i.e. a lift of $\pi$ to $\overline{\QQ}_l$.  Then the above
isomorphism takes $\tpi$ to the unique lift $\trho$ of $\rho$ such that
$\tomega^{-1}\det \trho$ corresponds to the central character of $\tpi$.  Since
the representation attached to $\tpi$ by (Tate normalized) local Langlands is
also a lift of $\trho$ with this determinant, the two coincide.

It remains to show the above isomorphism descends to an isomorphism
over $W(k)$.  Let $\sigma$ be an automorphism of $W(\overline{\FF}_l)$ over
$W(k)$.  Then the twist $(\pi_A)^{\sigma}$ of an $A$-deformation $\pi_A$ of $\pi$ is
a deformation of $\pi^{\sigma}$ over the $W(\overline{\FF}_l)$-algebra $A^{\sigma}$
whose underlying ring is $A$ but whose $W(\overline{\FF}_l)$-module structure
is twisted by $\sigma$.  The descent data isomorphism $\pi \cong \pi^{\sigma}$ allows us
to identify $(\pi_A)^{\sigma}$ with an $A^{\sigma}$-deformation of $\pi$.  This gives an
action of $\sigma$ on $R_{\pi}^{\univ}$.  As $\rho$ is also defined over $k$ we get a
similar action on $R_{\rho}^{\univ}$.  It suffices to show that conjugating the
isomorphism $R_{\pi}^{\univ} \cong R_{\rho}^{\univ}$ by $\sigma$ induces the same isomorphism;
by uniqueness it is enough to check this on $\overline{\QQ}_l$-points.  But this
is simply the (standard) fact that if $\tpi$ and $\trho$ correspond under (Tate normalized) 
local Langlands, then so do $\tpi^{\sigma}$ and $\trho^{\sigma}$.
\end{proof}

\end{document}